\documentclass[12pt]{amsart}
\usepackage{amssymb}
\usepackage{amsfonts}
\usepackage{latexsym}
\usepackage{amscd}
\usepackage[mathscr]{euscript}
\usepackage{xy} \xyoption{all}

\newdir{|>}{%
!/4.5pt/@{|}*:(1,-.2)@^{>}*:(1,+.2)@_{>}}

\newcommand{\N}{{\mathbb{N}}}
\newcommand{\Z}{{\mathbb{Z}}}

\newcommand{\uloopr}[1]{\ar@'{@+{[0,0]+(-4,5)}@+{[0,0]+(0,10)}@+{[0,0] +(4,5)}}^{#1}}
\newcommand{\uloopd}[1]{\ar@'{@+{[0,0]+(5,4)}@+{[0,0]+(10,0)}@+{[0,0]+ (5,-4)}}^{#1}}
\newcommand{\dloopr}[1]{\ar@'{@+{[0,0]+(-4,-5)}@+{[0,0]+(0,-10)}@+{[0, 0]+(4,-5)}}_{#1}}
\newcommand{\dloopd}[1]{\ar@'{@+{[0,0]+(-5,4)}@+{[0,0]+(-10,0)}@+{[0,0 ]+(-5,-4)}}_{#1}}

\newcommand{\luloop}[1]{\ar@'{@+{[0,0]+(-8,2)}@+{[0,0]+(-10,10)}@+{[0, 0]+(2,2)}}^{#1}}

\newtheorem{lem}{Lemma}[section]
\newtheorem{corol}[lem]{Corollary}
\newtheorem{theorem}[lem]{Theorem}
\newtheorem{proposition}[lem]{Proposition}
\newtheorem{remark}[lem]{Remark}
\newtheorem{definition}[lem]{Definition}

\newtheorem{example}[lem]{Example}

\begin{document}

\title[The socle series of a Leavitt path algebra]{The socle series of a Leavitt path algebra}
\author{Gene Abrams}\author{Kulumani M. Rangaswamy}
\address{Department of Mathematics, University of Colorado,
Colorado Springs CO 80933 U.S.A.} \email{abrams@math.uccs.edu}\email{krangasw@uccs.edu}
\author{Mercedes Siles Molina}
\address{Departamento de Algebra, Geometr\'ia, y Topolog\'ia, Universidad de M\'alaga,
29071 M\'alaga, Spain.}\email{msilesm@uma.es}
\thanks{The first  author is partially supported by the U.S. National Security Agency under grant number H89230-09-1-0066.   The
 third author is partially supported by the Spanish MEC and Fondos FEDER through project MTM2007-60333, and jointly by the Junta de Andaluc\'{\i}a and Fondos FEDER through projects FQM-336 and FQM-2467.  This work was initiated while the third author was hosted as a Research Visitor by the Department of Mathematics at the University of Colorado at Colorado Springs. She thanks this host center for its warm hospitality and support.} 

\begin{abstract}
We investigate the ascending Loewy socle series of Leavitt path algebras $L_K(E)$ for an arbitrary graph $E$ and field $K$.   We  classify those graphs $E$ for which $L_K(E)=S_{\lambda}$ for some element $S_{\lambda}$ of the Loewy socle series. We then show that for any ordinal $\lambda$ there exists a graph $E$ so that the Loewy length of $L_K(E)$ is $\lambda$.  Moreover, $\lambda \leq \omega $ (the first infinite ordinal)  if $E$ is a row-finite graph.
\end{abstract}

\maketitle

Graph $C^*$-algebras comprise an important class of $C^{\ast }$%
-algebras.  These  have been extensively studied over the past ten years
(see e.g. \cite{R}), and include as specific examples the Cuntz-Krieger algebras, the
Toeplitz algebra, and many of the AF-algebras. Recently the investigation of the algebraic aspects of the analogues
of these $C^{\ast }$-algebras, the {\it Leavitt path algebras} $L_K(E)$, have
become an active topic of research (see e.g. \cite{AA1}, \cite{AA2}, \cite{AA3},
\cite{AMP}, \cite{ABGM1}, \cite{ABGM2}, and \cite{APM}, among others).  This note, in which we study the ascending Loewy socle
series of $L_K(E)$ (see Definition \ref{socleseriesdefinition}), is one more contribution in
the study of the algebraic structure of Leavitt path algebras
over an arbitrary graph $E$ and arbitrary field $K$. Unlike many of the earlier articles written on the topic of Leavitt path algebras, here an {\it arbitrary} graph $E$ means there are neither  cardinality restrictions nor
graph-theoretical restrictions (e.g. row-finiteness)  imposed on  $E$.

W. Krull (\cite{K1}, \cite{K2}) was the first to develop the interrelationships between the socle series of a ring and its ideal structure; this was achieved in the context  of Noetherian rings.
Subsequently, in the case of rings satisfying minimum conditions modulo the Jacobson
radical, the socle series turned out to be quite useful (as shown, for example, by
Artin-Nesbitt-Thrall \cite{ANT} and Dauns \cite{D}).

Our goal in this article is to describe the socle series of a Leavitt path algebra $L_K(E)$ in terms of the graph $E$.   Specifically, we achieve two main results:   we give in Theorem \ref{Loewyring} an explicit description of those graphs $E$ for which the socle series terminates at $L_K(E)$ (following \cite{CF}, we call such a ring a {\it Loewy ring}); and in Theorem \ref{realizationtheorem} we construct, for each ordinal $\lambda$, a graph $E$ for which $L_K(E)$ is a Loewy ring of length $\lambda$.   In contrast, we show in Theorem \ref{rowfinitegiveslengthleqomega} that if $E$ is a  row-finite graph  then  the Loewy length of $L_K(E)$ must be finite or  $\omega$ (the first countable ordinal).   Further, we show that if $L_K(E)$ is a Loewy ring then necessarily $E$ is an acyclic graph, but not conversely. In particular, we conclude that, for any graph $E$, if $L_K(E)$ is a Loewy ring then $L_K(E)$ must be von Neumann regular.

\section{Basic concepts: Loewy rings, and the socle of $L_K(E)$}\label{basicsection}

In this first section we describe the ascending Loewy  socle series of an arbitrary ring, and then study this series in the context of Leavitt path algebras.   Along the way we provide a description of the socle of $L_K(E)$ for an arbitrary graph $E$.

The Loewy series of a module $M$ over a ring $R$ is defined in terms of
its socle. Recall that the \textit{socle} $Soc(M)$ of an $R$-module $M$\ is the
sum of all its simple submodules (where $Soc(M)$ is defined to be $\{0\}$ in case $M$ has no simple
submodules). In particular, if we consider $R$ as a left $R$-module, then the \textit{left
socle }$Soc_{l}(R)$ is the sum of all the simple left ideals of $R$. Similarly, one
defines the \textit{right socle} $Soc_{r}(R)$ of $R$. It is straightforward to
verify that both $Soc_{l}(R)$ and $Soc_{r}(R)$ are two-sided ideals of $R$. In
general the left socle of $R$ need not equal its right socle (for instance, $Soc_{l}(R) \neq Soc_{r}(R)$ for $R$ the ring of $n \times n$ upper triangular matrices over a field, $n\geq 2$).
However, for rings $R$ whose Jacobson radical is zero, it is known (see e.g. \cite{J}) that $Soc_{l}(R)=Soc_{r}(R)$; in this situation,
we shall denote the socle simply by $Soc(R)$.

The Jacobson radical $J(L_K(E))$ is shown to be $\{0\}$ in \cite{ABGM1} and \cite{ABGM2} for row-finite or countable graphs $E$.  As we shall demonstrate in Proposition \ref{zeroradicalingeneral}, in fact $J(L_K(E)) = \{0\}$ for any arbitrary graph $E$.  In particular, we get $Soc_{l}(L_K(E))=Soc_{r}(L_K(E))$.   (We note that the equality  $Soc_{l}(L_K(E))=Soc_{r}(L_K(E))$ will also follow from the internal description of the socle given in Theorem \ref{soclelinepoints}.)

\begin{definition}\label{socleseriesdefinition}

{\rm  Let $R$ be any (not necessarily unital) ring for which $J(R)=\{0\}$, and let $\tau=2^{|R|}$.    The {\it
ascending Loewy socle series} of $R$ is a well-ordered ascending chain of
two sided ideals
$$ 0=S_{0}<S_{1}<\cdot \cdot \cdot <S_{\alpha }<S_{\alpha +1}<\cdot \cdot \cdot
\qquad \qquad (\alpha <\tau )$$
where, for each $\alpha <\tau $,
$$S_{\alpha +1}/S_{\alpha }=Soc(R/S_{\alpha }) \ \  \mbox{if }\gamma = \alpha+1 \mbox{ is not a limit ordinal, and }$$
$$S_{\gamma }=\cup
_{\alpha <\gamma }S_{\alpha } \ \ \mbox{if } \gamma \mbox{ is a limit ordinal.}$$

\smallskip

\noindent
For each $\alpha < \tau$, $S_{\alpha }$ is called the {\it $\alpha$-th  socle of }$R$. The least ordinal $\lambda $ for which $S_\lambda = S_{\lambda + 1}$ is called the {\it  Loewy length} $\ell(R)$ {\it of} $R$.
We call $R$ a {\it Loewy ring} in case $R=S_{\alpha}$ for some $\alpha$.
We often refer to the  ascending Loewy socle series of a ring simply as its {\it socle series}.

}
\end{definition}

Using left (resp., right) socles one can define the notion of a left (resp., right)
Loewy ring in the analogous manner.

For a unital ring $R$ it is
straightforward to show that $R$ is a left Loewy ring if and only if every
nonzero left $R$-module has a nonzero socle. Rings with this
property have been studied under the name {\it left semiartinian rings} (see e.g. \cite{Bac}).

\smallskip

We now review the appropriate graph- and algebra-theoretic terminology used in this article.  For additional information about these ideas, see e.g. \cite{AA1}, \cite{R}, and \cite{T}.

A \emph{(directed) graph} $E=(E^{0},E^{1},r,s)$ consists of two
sets $E^{0},E^{1}$ and maps $r,s:E^{1}\rightarrow E^{0}$. (The sets $E^{0}$
and $E^{1}$ are allowed to be of arbitrary cardinality.) The elements of $%
E^{0}$ are called \emph{vertices} and the elements of $E^{1}$ \emph{edges}.
A \emph{path} $\mu $ in a graph $E$ is a sequence of edges $\mu =e_{1}\dots
e_{n}$ such that $r(e_{i})=s(e_{i+1})$ for $i=1,\dots ,n-1$. In this case, $%
s(\mu ):=s(e_{1})$ is the \emph{source} of $\mu $, $r(\mu ):=r(e_{n})$ is
the \emph{range} of $\mu $, and $n$ is the \emph{length} of $\mu $. We view
the elements of $E^{0}$ as paths of length $0$. If $\mu $ is a path in $E$,
and if $v=s(\mu )=r(\mu )$ and $s(e_{i})\neq s(e_{j})$ for every $i\neq j$,
then $\mu $ is called a \emph{cycle based at} $v$. If $s^{-1}(v)$ is a
finite set for every $v\in E^{0}$, then the graph is called \emph{row-finite}. We call a vertex $v$  a \emph{sink} if   $s^{-1}(v)$ is empty; an \emph{infinite emitter} if $s^{-1}(v)$ is an infinite set; and a \emph{regular vertex} otherwise.   A subset $H$ of $E^{0}$ is said to be \textit{hereditary} if whenever $u\in H$ and there is a path $\mu $\ with  $s(\mu )=u$ and $%
r(\mu )=w$ then $w\in H$. A subset $H$ of $E^0$ is said to be \textit{%
saturated} if, for any regular vertex $v\in E^{0}$, whenever $r(s^{-1}(v))\subseteq H$ then $v\in H$.  It is easy to see that the intersection of
hereditary saturated subsets is again hereditary saturated. Given a
hereditary subset $H$, the smallest hereditary saturated subset containing $H$ is called the \textit{saturated closure} of $H$, and is denoted by $\overline{H}$.

For any vertex $v\in E^{0}$, the \textit{tree of }$v$ is the set
$$T(v)=\{w\in E^{0} : \mbox{ there is a path }\mu \mbox{ with }s(\mu )=v \mbox{ and } r(\mu )=w\}.$$
(We note that $T(v)$ need not be a tree in the usual graph-theoretic sense, as $T(v)$ might contain cycles.)   A vertex $v$ is said to  \textit{have a bifurcation}, or  \textit{is a bifurcation vertex}, if  $|s^{-1}(v)|$ $>1$. \  Following \cite{ABGM1} and \cite{ABGM2},  we say a
vertex $v$ is a \textit{line point} if no vertex in  $T(v)$ is either a bifurcation vertex or is the base of a cycle. If $v$ is a line point, then
it is clear that every $w\in T(v)$ is also a line point; specifically, the set $L$ of line points in a
graph $E$ is a hereditary subset. Observe that  $v$
is a line point exactly when the vertices in $T(v)$ form a (finite or infinite)
line segment.

A ring $R$ is said to be {\it von Neumann regular} if for each $x\in R$ there exists $y\in R$ for which $x=xyx$.
A ring $R$ is said to be $\Z$-{\it graded} in case there is a decomposition $R = \oplus_{n\in \Z}R_n$ as abelian groups having the property that $R_m\cdot R_n \subseteq R_{m+n}$ for all $m,n\in \Z$.  For a $\Z$-graded ring $R$ and $0\neq x\in R$, write $x = \sum x_n$ with $x_n \in R_n$ for each $n$.  The {\it degree} of $x$ is the maximum $n$ for which $x_n \neq 0$.   An ideal $I$ of a $\Z$-graded ring $R$ is {\it graded} in case $I = \oplus_{n\in \Z}(I \cap R_n)$;  equivalently, $I$ is graded if whenever $x_n\in R_n$ for each $n$ has $\sum x_n \in I$, then $x_n \in I$ for each $n$.  In this case, observe that the ring $R/I$ inherits a natural $\Z$-grading, and that the natural map $R \rightarrow R/I$ is a graded homomorphism.   Moreover, for an ideal $J$ of $R$ having $J \supseteq I$, $J$ is a graded ideal of $R$ if and only if $J/I$ is a graded ideal of $R/I$.

\medskip

\begin{definition}\label{definition}
 \rm{Let $E$ be any directed graph, and $K$ any field.
The \emph{Leavitt path $K$-algebra} $L_K(E)$ \emph{of $E$ with coefficients
in $K$} is the $K$-algebra generated by a set $\{v : v\in E^0\}$ of
pairwise orthogonal idempotents, together with a set of variables $%
\{e,e^* : e\in E^1\}$, which satisfy the following relations:

(1) $s(e)e=er(e)=e$ for all $e\in E^1$.

(2) $r(e)e^*=e^*s(e)=e^*$ for all $e\in E^1$.

(3) $e^*e^{\prime }=\delta _{e,e^{\prime }}r(e)$ for all $e,e^{\prime} \in E^1$.

(4) $v=\sum _{\{ e\in E^1: s(e)=v \}}ee^*$ for every regular
vertex $v\in E^0$.

\smallskip
\noindent
When the role of $K$ is not central to the discussion, we sometimes denote $L_K(E)$ simply by $L(E)$.

}
\end{definition}

\medskip


For any $e\in E^{1}$ we  let $r(e^{\ast })$ denote $s(e)$, and we let $%
s(e^{\ast })$ denote $r(e)$. If $\mu =e_{1}\dots e_{n}$ is a path, then we
denote by $\mu ^{\ast }$ the element $e_{n}^{\ast }\cdots e_{1}^{\ast }$ of $%
L_{K}(E)$.

\begin{remark}\label{mumustarlinepoint}
{\rm In particular, by Property (3) we have that if $w$ is any vertex in $E$, and $\mu$ is any path for which $r(\mu)=w$, then $\mu^* \mu = w$.   On the other hand, if $v$ is a vertex having only one edge $e$ in $s^{-1}(v)$, then Property (4) yields $ee^* =v$.  This observation and an easy induction yields  the following important property of the line points of $E$:  if $v$ is a line point, and $\mu$ is a path for which $s(\mu)=v$,    then $\mu \mu^* = v$.}
\end{remark}

If $E$ is a graph for which $E^{0}$ is finite then we have $\sum_{v\in
E^{0}}v$ is the multiplicative identity in $L_{K}(E)$; otherwise, $L_{K}(E)$
is a ring with a set of local units consisting of sums of distinct vertices.   In particular, if $I$ is any ideal of $L_K(E)$, then $I=L_K(E)$ if and only if $I\cap E^0 = E^0$.
Conversely, if $L_{K}(E)$ is unital, then $E^{0}$ is finite.

$L_{K}(E)$ is a
$\Z$-graded $K$-algebra, spanned as a $K$-vector space by $%
\{pq^{\ast }: p,q$ are paths in $E\}$. (Recall that the elements of $E^{0}
$ are viewed as paths of length $0$, so that this set includes elements of
the form $v$ with $v\in E^{0}$.)
In particular, for each $n\in \Z$,
the degree $n$ component $L_{K}(E)_{n}$ is spanned by elements of the form $%
\{pq^{\ast }: \mathrm{length}(p)-\mathrm{length}(q)=n\}$.



In the following result we present a useful description of the hereditary saturated closure of a hereditary subset of a graph.

\begin{lem}
Let $H$ be a hereditary subset of vertices in a graph $E$. Then a
vertex $u$ belongs to the saturated closure $\overline{H}$ of $H$ if and
only if there exists a positive integer $n$ such that every path of length $\geq n$ in $E$ that begins with $u$ must end in a vertex of $H$.
\end{lem}

\begin{proof}
By definition, $\overline{H}=\cup _{n<\omega }H_{n}$ where $H_{n}$ is
defined inductively as follows.   Let $H_{0}=H$. If   $H_{k}$ has been defined
for some $k\geq 0$, then set
$$H_{k+1}=H_{k}\cup \{v\in E^{0}:v \mbox{ is
a regular vertex, and } r(s^{-1}(v))\subset H_{k}\}.$$
We show, by induction on $k$, that  $v\in H_{k}$ if and only if every
path of length $\geq k$ that begins with $v$ must end in a vertex belonging
to $H$. This clearly holds when $k=0$, since $H_{0}=H$ is hereditary.  Assume we have proved the result when $k=m\geq 0$.  Let $v\in
E^{0}$ be a regular vertex, say
$s^{-1}(v)= \{e_{1},...,e_{t}\}$ and $r(e_{i})=v_{i}$, for $i=1,...,t$. Now $v\in
H_{m+1}$ precisely when $v\in H_{m}$ or when $\ r(s^{-1}(v))=\{v_{1},...,v_{t}\}\subset H_{m}$. By
induction, this is equivalent to saying that every path of length at least $m$
that begins with any vertex $v_{i}$ ends in $H$. This is clearly equivalent
to requiring that every path of length at least $m+1$ that begins with $v$ must
end in a vertex belonging to $H$. Hence the result follows.
\end{proof}

\vspace{1pt}

We continue by describing the following concepts which were initially introduced in papers dealing with C$^*$-algebras (see e.g. \cite{R}). These ideas were also considered for row-finite graphs in \cite{APM}, and for not-necessarily row-finite graphs in \cite{T}.

\begin{definition}\label{breakingvertices}
{\rm Let $E$ be an arbitrary graph and let $H$ be a hereditary saturated
subset of vertices in $E$.  \

(i) The set of  \textit{breaking vertices} of $H$, denoted $B_H$, is the set
\begin{equation*}
B_{H}=\{v\in E^{0}\backslash H:v\text{ is an infinite emitter, and }
0<|s^{-1}(v)\cap r^{-1}(E^{0}\backslash H)|\text{ }<\infty \}\text{. }
\end{equation*}

 (ii) The \textit{quotient graph} $E|H$ is defined as follows.   Let $
B_{H}^{\prime }$ be a set which is in one-to-one correspondence with $B_{H}$,
and write $B_{H}^{\prime }=\{v^{\prime }:v\in B_{H}\}$.
Define
$$(E|H)^{0}=(E^{0}\setminus H)\cup B_{H}^{\prime } \  \ \ \   \ \mbox{and} \  \ \ \  \ \
(E|H)^{1}=\{e\in E^{1}:r(e)\notin H\}\cup \{e^{\prime }:e\in E^{1} \mbox{ with } r(e)\in
B_{H}\}.$$
The source and range functions  $s_{E|H}$ and $r_{E|H}$ coincide with the functions $s_E$ and $r_E$ when applicable, while we define
$s_{E|H}(e^{\prime })=s_E(e)$ and $r_{E|H}(e^{\prime })=(r_E(e))^{\prime }$.

}
\end{definition}

We note that  each $v^{\prime }$\ is a sink in the graph $E|H$,  and so
is a line point in $E|H$.

\medskip

A useful tool in our construction is the following important theorem  of
Tomforde (\cite[ Theorem 5.7]{T}). This theorem has been established  under the hypothesis
that $E$ is a graph with at most countably many vertices and edges; however, an
examination of the proof reveals that the countability condition on $E$ is
not utilized. Hence we give a reformulation, tailored to our needs,  of
 various parts of  \cite[Lemma 5.6 and Theorem 5.7]{T} for arbitrary graphs $E$.

 \begin{definition}
 {\rm Let $E$ be a graph, and $S$ any subset of $E^0$.   We denote by $I(S)$ the two-sided ideal of $L(E)$ generated by $S$.}
 \end{definition}

\begin{theorem}\label{Tomfordetheorem}
Let $E$ be an arbitrary graph, and let $H$ be a hereditary saturated subset of $E$.  Then

(i)  $I(H)$ is a graded ideal of $L_K(E)$.

(ii) There is an algebra epimorphism
$$\phi
:L_{K}(E)\rightarrow L_{K}(E|H)$$
 for which $\ker \phi =I(H)$.  In particular,
 $$L_{K}(E)/I(H)\cong L_{K}(E|H).$$
Moreover, for each $v\in B_{H}$
 we have $\phi (v-\sum_{s(e)=v,r(e)\notin H}ee^{\ast })=v^{\prime
}$ (where $v^{\prime}$ is described in Definition \ref{breakingvertices}).

\smallskip

(iii) Let  $I$ be a graded ideal of $L_{K}(E)$. If $I\cap E^{0}=H$, and we define
$$S=\{v\in B_{H}:v-\sum_{s(e)=v,r(e)\notin H}ee^{\ast } \in I\},$$
 then $I = I(H\cup S).$



\end{theorem}

 As an application of
Theorem \ref{Tomfordetheorem}, we get as promised the following extension to arbitrary graphs of a result  proved in
 \cite{AA3} and \cite{ABGM2} for row-finite or countable graphs.

\begin{proposition}\label{zeroradicalingeneral}
 Let $E$ be an arbitrary graph. Then the Jacobson radical $J(L(E))=\{0\}$.  In particular, $L(E)$ contains no nonzero nilpotent left or right ideals.
\end{proposition}

\begin{proof}
Since $L(E)$ is a ring with local units,    $J=J(L(E))$ is a
graded ideal by \cite[Lemma 6.2]{AA3}.   From Theorem \ref{Tomfordetheorem}(iii), $J$  is the ideal of $L(E)$
generated by $H\cup S$, where $H=E^{0}\cap J$ and $S=$ $\{v\in
B_{H}:v-\sum_{s(e)=v,r(e)\notin H}ee^{\ast }$ $\in J\}$.   Since  $J$ contains no nonzero idempotents, $H$ must be the empty
set. From the definition of $B_{H}$ we then conclude that $B_{H}$, and hence
$S$, must be the empty set as well. This implies that $J=\{0\}.$
\end{proof}

 When $R=L_{K}(E)$ and $E$ is row-finite or is countably infinite, then $Soc(R)$ has been described in \cite{ABGM1} and \cite{ABGM2}. Specifically, $Soc(R)$ is the
ideal of $L_K(E)$  generated by the line points of $E$.  Our goal for the remainder of this section is to extend this result to the case where $E$ is an
arbitrary graph.  Along the way, we provide a simpler proof of the fact that the line
points are precisely those vertices which generate a simple left (equivalently, right)
ideal of $L_{K}(E)$.

\begin{proposition}(\cite{ABGM1}, \cite{ABGM2}) \label{linepoint}
Let $E$ be an arbitrary graph and $v\in E^{0}$. \ Then $v$ is a line point
exactly when the left ideal $L_{K}(E)v$ of $L_K(E)$ is simple.
\end{proposition}

\begin{proof}
Suppose $v$ is a line point. We claim that every nonzero $L_K(E)$-endomorphism
of $L_{K}(E)v$ is an automorphism. Since $v$ is an idempotent, it is
well-known  (and can also be easily seen) that ${\rm End}(L_{K}(E)v)\cong
vL_{K}(E)v$.  An arbitrary element $a\in vL_{K}(E)v$ will be of the form $
a=v(\sum_{i=1}^nk_{i}p_{i}q_{i}^{\ast })v=$ $\sum_{i=1}^n
k_{i}(vp_{i}q_{i}^{\ast }v)$. \ Observe that if $vp_{i}q_{i}^{\ast }v\neq 0$ for some $1 \leq i \leq n$,  then $s(p_{i})= s(q_{i})=v,$ and $r(p_i)=r(q_i)$.  Since $v$
is a line point, this yields that $p_{i}=q_{i}$.
 But by Remark \ref{mumustarlinepoint} we then get $vp_{i}q_{i}^{\ast }v=vp_{i}p_{i}^{\ast }v=v$. Hence $a=(\sum_{i=1}^n k_{i})v$
and we conclude that $vL_{K}(E)v=Kv$ is a field with identity element $v$, which establishes the claim.

Since $L_K(E)$ has local units, we have $L_K(E)a \neq \{0\}$. By Proposition \ref{zeroradicalingeneral}, $L_K(E)$ has no nonzero nilpotent one-sided ideals, so that  $(L_K(E)a)^2 \neq \{0\}$.  In particular there is an element  $b\in L_{K}(E)$ such that  $aba\neq 0$.  Thus right multiplication by $ba$ is a
nonzero endomorphism of $L_{K}(E)v$ \ (as \ $aba\neq 0$), and hence by the previous paragraph is an
automorphism. In particular $v=cba$ for some $c\in L_{K}(E)$. This means that $v\in
L_{K}(E)a$, showing that $L_{K}(E)v$ is simple.

Conversely,  suppose  $L_{K}(E)v$ is simple.  Suppose by way of
contradiction that $T(v)$ has vertices with bifurcations, and choose a
bifurcation vertex $u\in T(v)$ so that there is a path $\mu $\ of shortest
length connecting $v$ to $u$.  Then  $L_{K}(E)u$ is simple,  since using Remark \ref{mumustarlinepoint} it is easy to show that
right multiplication by $\mu$ induces an isomorphism from $L_{K}(E)v$ to  $L_{K}(E)u$ (with inverse map right multiplication by $\mu^*$).  As shown in \cite[Proposition 2.5]{ABGM1}, the existence of a cycle based at a
vertex in $T(v)$ would contradict the simplicity of $L_{K}(E)v$, so we may assume that $T(v)$ is acyclic.  Now let $e$ be an edge with $s(e)=u$. Then $L_{K}(E)u=L_{K}(E)ee^{%
\ast }\oplus C$, where $C=\{x-xee^{\ast } : x\in
L_{K}(E)\}$.  Now $C\neq \{0\}$, since there is another edge $f\neq e$ with $%
s(f)=u$ (and $fe=0$ since $f$ is not a loop by the previous observation), so that $f=f-fee^{\ast }\in C$. This contradicts the
simplicity of $L_{K}(E)u$. Hence $T(v)$ contains no vertices with bifurcations.    Hence $v$ must
be a line point.
\end{proof}

\begin{remark}
{\rm Although it would also follow from more general (and much deeper) structural results about semiprime rings, in fact a direct ``left to right" modification of the proof given above easily yields that $v\in E^0$ is a line point if and only if the right ideal $vL_K(E)$ of $L_K(E)$ is simple.}
\end{remark}

 The following result is established in \cite[Theorem 4.2]{ABGM1} for row-finite graphs, and in \cite[Theorem 5.2]{ABGM2} for countable graphs; effectively the proof is the same in both cases.  Using Proposition \ref{linepoint}, indeed this
 same proof can be used to establish the result for arbitrary graphs $E.$  That is, we have

\begin{theorem}\label{soclelinepoints}  For an arbitrary graph $E$ and field $K$, $Soc(L_{K}(E))$ is the
two sided ideal generated by the set of line points in $E$.
\end{theorem}

\bigskip

 Theorems \ref{Tomfordetheorem} and \ref{soclelinepoints} provide us with the two fundamental tools we will use to establish our main results.    We seek to understand the behavior of the socle series of $L(E)$; this requires us to understand the behavior of the socle of each of the quotients $L(E)/S_{\alpha}$.   By Theorem \ref{Tomfordetheorem} we will be able to realize $L(E)/S_{\alpha}$ as $L(F)$ for some graph $F$.   In turn,  by Theorem \ref{soclelinepoints}, we will be able to identify the socle of this quotient in terms of the line points of $F$.

\section{\protect\bigskip Examples}\label{ExamplesSection}

In this section we analyze the Loewy lengths of  Leavitt path algebras $L_{K}(E)$ of
various graphs $E$ over any field $K$.  We begin with the most basic type of Leavitt path algebras.

\begin{example}\label{acyclicexample}
{\rm

Let $E$ be a finite acyclic graph.  Then $L_K(E)$ is semisimple artinian, as $L_K(E)$ is a direct sum of complete
matrix rings over $K$ (see e.g. \cite[Proposition 3.5]{AAS}).
Thus we  have in this case that $L_K(E)$ is a Loewy ring, and $\ell(L_K(E))=1$.   In particular, the trivial graph
$$P_0 : \ \ \  \ \bullet^v$$
consisting of one vertex and no edges is a Loewy ring of length 1.
}
\end{example}

\begin{example}\label{notLoewyexamples}
{\rm    In this second example we present Leavitt path algebras which are not Loewy rings.   First, if $E$ is the graph with a single vertex $v$\ and a single loop $x$\ at $%
v$,  then  $R=L_{K}(E)\cong K[x,x^{-1}]$, the ring of Laurent polynomials
over $K$.  In this case  $\ell(R)=0$ since $Soc(R)=\{0\}$.   Seen another way, since $E$ has no line
points,  $Soc(R)=\{0\}$ by Theorem \ref{soclelinepoints}.   In particular, $R$ is not a Loewy ring.

 Now let $T$ be the graph
 \medskip
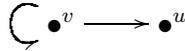
\[\xymatrix{ {\bullet}^{v} \ar@(ul,dl) \ar[r]&{\bullet}^u} \]
 Since $u$ is the unique line point of $T$, we get that  $R=L_{K}(T)$ satisfies $S_{1}=Soc(R)=I(u)$, the ideal generated by $u$.  It is easy to show that $Soc(R)\neq R$, since in particular $v\notin Soc(R)$.  Also, by Theorem \ref{Tomfordetheorem}(ii),  $R/S_{1}\cong L_{K}(T|\{u\})$. Since  $T|\{u\}$ is the graph with a single vertex and single edge, we obtain from
the preceding paragraph that $R/S_{1}\cong K[x,x^{-1}]$ and that $Soc(R/S_{1})=\{0\}$.
 Hence $S_{1}=S_{2}=\cdot
\cdot \cdot $.  Specifically, we have $\ell(R)=1$, but $R$ is not a Loewy ring as $R\neq S_{i}$ for any $i$.
}
\end{example}

\begin{example}\label{infiniteclockexample}
{\rm Let $\aleph$ be any infinite cardinal, and let $C_{\aleph}$ be the ``infinite clock" graph with $\aleph$ edges
 $$ 
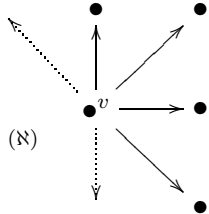
\xymatrix{ & {\bullet} & {\bullet} \\  & {\bullet}^v  \ar@{.>}[ul] \ar[u] \ar[ur] \ar[r] \ar[dr] \ar@{.>}[d]
\ar@{}[dl] _{(\aleph)} & {\bullet} \\ &  & {\bullet}}$$

The set $H=\{r(e):e\in (C_{\aleph })^{1}\}$ is  a hereditary
saturated subset of $(C_{\aleph })^{0}$ (recall that by definition the saturated condition only applies at regular vertices), and indeed is precisely the set of line
points in $C_{\aleph }$. By Theorem \ref{soclelinepoints}, $Soc(L_{K}(C_{\aleph }))=I(H)$. Then, by Theorem \ref{Tomfordetheorem}(ii), $L_{K}(C_{\aleph})/I(H)\cong L_{K}(C_{\aleph }|H)=L_{K}(\{v\})\cong K$. Since $Soc(K)=K$, we get that $L_{K}(C_{\aleph })$ is thus a Loewy ring with Loewy length $2$.
}
\end{example}

\begin{example}\label{disconnectedexample}
{\rm
Let $P_0^{\omega}$ denote the graph
consisting of countably many vertices  and no edges.  Then $L_K(P_0^{\omega})\cong \oplus_{n \in \N}K$ as (nonunital) rings, so that we immediately conclude that $L_K(P_0^{\omega})= Soc(L_K(P_0^{\omega}))$.  In particular, $L_K(P_0^{\omega})$ is a Loewy ring of Loewy length 1.   (Note that every vertex in $P_0^{\omega}$ is vacuously a line point, so that this result also follows immediately from Theorem \ref{soclelinepoints}.)
}
\end{example}

\begin{example}\label{infinitelineexample}
{\rm
Let $P_1$ denote the ``infinite line" graph

$$\xymatrix{ P_1: &  {\bullet}^{v_{1,1}} \ar[r]  & {\bullet}^{v_{1,2}} \ar[r]  & {\bullet}^{v_{1,3}} \ar[r]  &
               {\bullet}^{v_{1,4}} \ar@{.>}[r] & }$$

\medskip

\noindent
Then clearly every vertex in $P_1$ is a line point, so by Theorem \ref{soclelinepoints} we conclude that  $L_K(P_1) = Soc(L_K(P_1))$.  In particular, $L_K(P_1)$ is a Loewy ring of Loewy length 1.
}
\end{example}

The graphs $P_0$ and  $P_1$ of the previous examples will be used as the foundation for the construction presented in the proof of Theorem \ref{realizationtheorem}.   Intuitively, $P_0$ and $P_1$ provide the two canonical examples of graphs for which the corresponding Leavitt path algebras are Loewy rings of Loewy length 1.

\begin{example}\label{P2example}
{\rm
Let $P_2$ denote the graph

$$
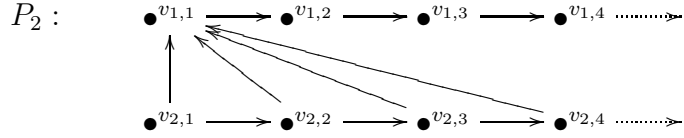
\xymatrix{ P_2: & {\bullet}^{v_{1,1}} \ar[r]  & {\bullet}^{v_{1,2}} \ar[r]  & {\bullet}^{v_{1,3}} \ar[r]  &
               {\bullet}^{v_{1,4}} \ar@{.>}[r] &  \\
            & {\bullet}^{v_{2,1}} \ar[r] \ar[u] & {\bullet}^{v_{2,2}} \ar[r] \ar[ul] & {\bullet}^{v_{2,3}} \ar[r] \ar[ull] & {\bullet}^{v_{2,4}}\ar@{.>}[r] \ar[ulll] & }$$

\medskip

\noindent
The set $H = \{v_{1,j} : j\in \N\}$ (i.e., the ``top row" of $P_2$) is the set of line points in the graph
$P_2$; furthermore, an easy observation yields that  $H=\overline{H}$.  Note that the quotient graph $P_2|H$ consists of the vertices and edges in the ``bottom row" of $P_2$, which is clearly isomorphic as a graph to $P_1$.  Now by Theorem \ref{soclelinepoints}, $Soc(L_{K}(P_{2}))=I(H)$, the ideal generated
by $H$.  Furthermore, by Theorem \ref{Tomfordetheorem}(ii),
$$L_{K}(P_2)/Soc(L_{K}(P_2))=L_{K}(P_2)/I(H)\cong L_{K}(P_2|H) \cong
L_{K}(P_1).$$
 We conclude that $L_K(P_2)$ is a Loewy ring having $\ell (L_{K}(P_2))=2$.
}
\end{example}

\begin{example}\label{downpyramids}
{\rm
For each integer $i\geq 2$ we construct inductively the ``pyramid" graph $P_i$ as follows.   The graphs $P_1$ and $P_2$ are presented in the preceding examples.  For each $i\geq 2$ we construct the graph $P_{i+1}$ from the graph $P_i$ by adding vertices $v_{i+1,1},v_{i+1,2},v_{i+1,3},...$ and two sets of edges: for each $j\in \N$, an edge from $v_{i+1,j}$ to $v_{i+1,j+1}$; and, for each $j\geq 1$, a single edge from $v_{i+1,j}$ to $v_{i,1}$.  So, for example, we have

$$
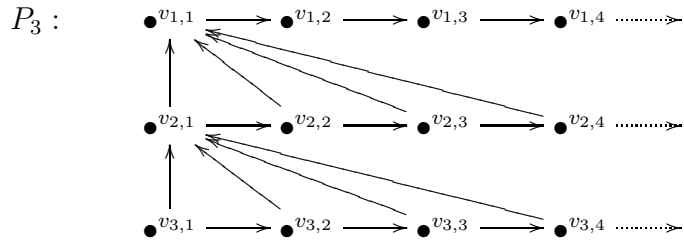
\xymatrix{ P_3: & {\bullet}^{v_{1,1}} \ar[r]  & {\bullet}^{v_{1,2}} \ar[r]  & {\bullet}^{v_{1,3}} \ar[r]  &
               {\bullet}^{v_{1,4}} \ar@{.>}[r] &  \\
             & {\bullet}^{v_{2,1}} \ar[r] \ar[u] & {\bullet}^{v_{2,2}} \ar[r] \ar[ul] & {\bullet}^{v_{2,3}} \ar[r] \ar[ull] & {\bullet}^{v_{2,4}}\ar@{.>}[r] \ar[ulll] &  \\
             & {\bullet}^{v_{3,1}} \ar[r] \ar[u] & {\bullet}^{v_{3,2}} \ar[r] \ar[ul] & {\bullet}^{v_{3,3}} \ar[r] \ar[ull] & {\bullet}^{v_{3,4}}\ar@{.>}[r] \ar[ulll] &  \\
            }$$

\medskip

\noindent
By induction, using the argument of Example \ref{P2example}, it is straightforward to show, for each $i\in \N$, that $L_K(P_i)$ is a Loewy ring for which $\ell(L_K(P_i))=i$.

We now view  $P_{i}\subseteq P_{i+1}$ for each $i\in \N$, and make two observations.  First, it is clear that $P_i^0$ (the set of vertices of $P_i$) is a hereditary saturated subset of $P_{i+1}$.  (Notice that although each vertex $v_{i+1, n}$ emits an edge into $P_i^0$, each such vertex as well emits an edge whose range vertex is {\it not} in $P_i^0$, whereby the saturated property of $P_i^0$ follows.)  In particular we can form the quotient graph $P_{i+1}|P_i^0$; it is immediate that $P_{i+1}|P_i^0 \cong P_1$ as graphs.    Second, with this inclusion of graphs we can then form the
graph
$$P_{\omega} = \cup_{i< \omega}P_i.$$
Again invoking the argument of Example \ref{P2example}, one can similarly show that  $L_K(P_{\omega})$ is a Loewy ring, and $\ell(L_K(P_{\omega}))=\omega$.

}
\end{example}

\begin{example}\label{uppyramids}
{\rm

We consider here a construction which looks quite similar to that achieved in Example \ref{downpyramids}, but which features an interesting twist.   For each $i\in \N$ we define the graph $Q_i$ as pictured here.

$$
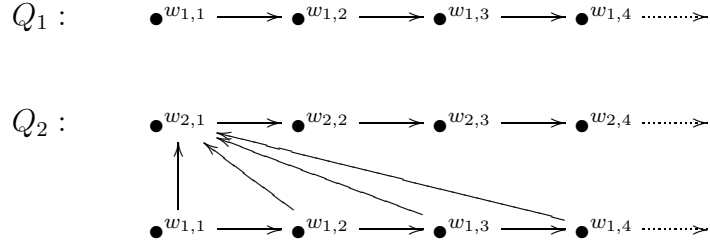
\xymatrix{ Q_1: &  {\bullet}^{w_{1,1}} \ar[r]  & {\bullet}^{w_{1,2}} \ar[r]  & {\bullet}^{w_{1,3}} \ar[r]  &
               {\bullet}^{w_{1,4}} \ar@{.>}[r] & }$$

\bigskip

$$\xymatrix{ Q_2: & {\bullet}^{w_{2,1}} \ar[r]  & {\bullet}^{w_{2,2}} \ar[r]  & {\bullet}^{w_{2,3}} \ar[r]  &
               {\bullet}^{w_{2,4}} \ar@{.>}[r] &  \\
            & {\bullet}^{w_{1,1}} \ar[r] \ar[u] & {\bullet}^{w_{1,2}} \ar[r] \ar[ul] & {\bullet}^{w_{1,3}} \ar[r] \ar[ull] & {\bullet}^{w_{1,4}}\ar@{.>}[r] \ar[ulll] & }$$

\medskip

\noindent
For each $i\geq 2$ we construct the graph $Q_{i+1}$ from the graph $Q_i$ by adding vertices $w_{i+1,1},w_{i+1,2},$ $w_{i+1,3},...$ and two sets of edges, as follows.  For each $j\in \N$, we add an edge from $w_{i+1,j}$ to $w_{i+1,j+1}$; and, for each $j\geq 1$, we add a single edge from $w_{i,j}$ to $w_{i+1,1}$.   Clearly for each $i\in \N$ the graph $Q_i$ is isomorphic to the graph $P_i$ of Example \ref{downpyramids}.   In particular,  for each $i\in \N$ we have that $L_K(Q_i)$ is a Loewy ring, and $\ell(L_K(Q_i))=i$.

Here is where the two examples diverge. If we  view $Q_{i}\subseteq Q_{i+1}$ for each $i\in \N$, then we can form the
graph $Q_{\omega} = \cup_{i< \omega}Q_i$.  Unlike its counterpart $P_{\omega}$, the graph $Q_{\omega}$ contains no line points.  In particular, $Soc(L_K(Q_{\omega})) = \{0\}$, so that $S_{\alpha}=\{0\}$ for all $\alpha$.   Thus, unlike its counterpart $L_K(P_{\omega})$, the Leavitt path algebra  $L_K(Q_{\omega})$ is not a Loewy ring.

}
\end{example}


\begin{remark}
{\rm
For an arbitrary ring $R$, dual to the left (resp., right) ascending Loewy socle  series is the corresponding left (resp., right) {\it descending Loewy radical series}.  Briefly, this is defined by setting $R_0$ to be $R$, setting $R_{\alpha +1}$ to be the intersection of all
maximal left (resp., right) submodules of $R_{\alpha }$,  and, for any limit ordinal $\gamma$, setting $R_{\gamma }$ to be $\cap _{\alpha
<\gamma }R_{\alpha }.$   (If $R_{\alpha }$ has no maximal submodules, then set $R_{\alpha }=R_{\alpha +1}$.)  Specifically, $R_{1} = J(R)$.  But
$J(L_K(E))=\{0\}$ for any graph $E$ by Proposition \ref{zeroradicalingeneral}, so that $R_1 = \{0\}$.  Thus the  descending Loewy radical
series is of little interest in the context of Leavitt path algebras.
}
\end{remark}

\section{Leavitt path algebras of arbitrary Loewy length}\label{MainSection}


The goal of this section is to describe in graph-theoretic terms  a necessary and sufficient condition on an arbitrary graph $E$
in order that the corresponding Leavitt path algebra $L_{K}(E)$ is a
Loewy ring; we achieve this characterization in Theorem \ref{Loewyring}.  As a result of this description, we are able to construct, for each ordinal $\lambda$, a graph $E$ with the property that $L_K(E)$ is a Loewy ring of length $\lambda$ (Theorem \ref{realizationtheorem}).

Recall that for a hereditary subset $S$ of vertices of a graph $E$, the set $\overline{S}$ denotes the saturated closure of $S$ in $E^0$.

\begin{definition}\label{Valphadefinition}
{\rm Let $E$ be an arbitrary graph and consider the Leavitt path algebra $L_{K}(E)$.  For each ordinal $\gamma$ we define transfinitely a hereditary saturated
subset $V_{\gamma}$ of $E^{0}$ as follows.

(1) \ $V_{1}$ is the saturated closure of the set of line points of $E.$

\smallskip
\noindent
Suppose $\gamma >1$ is any ordinal
and that the sets $V_{\alpha }$ have been defined for all $\alpha <\gamma $.

\smallskip

(2) \ If $\gamma $ is a limit ordinal, then $V_{\gamma }=\cup _{\alpha
<\gamma }V_{\alpha }$.

(3) \ If $\gamma =\alpha +1$ is a non-limit ordinal, then  $V_{\gamma}=E^{0}\cap I$, where $I$ is the ideal of $L_K(E)$ generated by the set
$$V_{\alpha } \ \ \cup \ \ \{w \in E^{0}\backslash V_{\alpha } \text{ \ : every
bifurcation vertex }u\in T_{E}(w)\backslash V_{\alpha }
\text{ has at most one edge }e$$
$$ \hskip-1in \text{ with }s(e) =u\text{ and }r(e)\text{ }%
\notin V_{\alpha }\}$$
$$\hskip-3.4in \ \cup \ \ \{v-\sum_{\substack{ s(e)=v  \\ %
r(e)\notin V_{\alpha }}}ee^{\ast }\text{ : }v\in B_{V_{\alpha }}\}\text{.}$$
}
\end{definition}

\smallskip

\begin{theorem}\label{Loewyring}
Let $E$ be an arbitrary graph and $K$ any field. For each ordinal $\alpha$
let  $S_{\alpha}$ denote the $\alpha$-th socle of $L_K(E)$, and let $V_{\alpha}$ denote the subset of $E^{0}$ given in Definition \ref{Valphadefinition}.  Then

(1)  $S_\alpha$ is a graded ideal of $L_K(E)$ for each $\alpha$.

(2)  $V_\alpha =  E^0 \cap S_\alpha $ for each $\alpha$.

(3)  $L_K(E)/S_\alpha \cong L(E|V_\alpha)$ as graded $K$-algebras for each $\alpha$.

(4)  $L_{K}(E)$ is a Loewy ring of length $\lambda $ if and only if $\lambda $ is the smallest ordinal such that $E^{0}=V_{\lambda }$.
\end{theorem}

\begin{proof}


We establish statements (1) and (2) simultaneously by transfinite induction.

When $\gamma = 1$, $V_{1}$ has been defined to be the saturated closure of the set of all line
points in $E$; so  $V_1$ is a hereditary saturated subset of $E^0$.   By Theorem \ref{soclelinepoints},
$S_{1} = Soc(L_K(E))$ is the two-sided ideal of $L_K(E)$ generated by $V_{1}=E^{0}\cap S_{1}$, and hence is a graded ideal of $L_K(E)$ by Theorem \ref{Tomfordetheorem}(i).

Now suppose $\gamma > 1$ and, for all $\alpha < \gamma$, we have that $V_{\alpha }=E^{0}\cap S_{\alpha }$ and that $S_{\alpha}$ has been shown to be a graded ideal.    Recall (see Definition \ref{breakingvertices}) that
$E|V_{\alpha }$ is the graph with
$$(E|V_{\alpha })^{0}= (E^{0}\backslash V_{\alpha })\cup B_{V_{\alpha }}^{\prime }
\ \ \  \mbox{ and} \ \ \  (E|V_{\alpha })^{1}=\{e\in E^{1} : r(e)\notin V_{\alpha }\}\cup
\{e^{\prime }:e\in E^{1}:r(e)\in B_{V_{\alpha }}\}.$$
Also, $s(e),r(e)$ are
defined as in $E$ if $e\in E^{1}$, while  $s(e^{\prime })=s(e)$ and $r(e^{\prime })=(r(e))^{\prime }$.
Then by  Theorem \ref{Tomfordetheorem}(ii), there is
an epimorphism $\phi: L_{K}(E)\longrightarrow L_{K}(E|V_{\alpha })$ having  ${\rm ker}(\phi) =S_{\alpha }$, for which  $L_{K}(E)/S_{\alpha }\cong L_{K}(E|V_{\alpha })$.

Suppose $\gamma$ is not a limit ordinal (so that $\gamma = \alpha + 1$ for some $\alpha$), and suppose $V_{\alpha }\neq E^{0}$ (and thus $S_{\alpha }\neq L_{K}(E)$).  Then define
$$V_{\alpha +1}^{\prime } = \{w\in E^{0}\backslash V_{\alpha }:\text{ every
bifurcation vertex \ }u\in T_{E}(w)\backslash V_{\alpha }\text{  has at
most one edge }e $$
$$ \hskip1in \text{ with }s(e)=u\text{ and }r(e)\text{ }\notin V_{\alpha
}\} $$
$$\hskip-3in \cup\ \  \{v-\sum_{\substack{ s(e)=v \\ r(e)\notin V_{\alpha }}}ee^{\ast }\text{
}\text{: } v\in B_{V_{\alpha }}\}\text{. }
$$
From Theorem \ref{Tomfordetheorem}(ii) we have that
$$\phi (v-\sum\limits_{s(e)=v,r(e)\notin V_{\alpha }}ee^{\ast })=v^{\prime }.$$
As noted previously, $v'$ is a sink  and hence a line point in $E|V_{\alpha }$. It is
then easy to check that $\phi (V_{\alpha +1}^{^{\prime }})$ is exactly the set of
line points of $E|V_{\alpha }$. So if $I$ is the ideal of $L_K(E)$  generated by $%
S_{\alpha }\cup V_{\alpha +1}^{\prime }$, then we get that $%
Soc(L_{K}(E|V_{\alpha }))\cong I/S_{\alpha }=Soc(L_{K}(E)/S_{\alpha })$.
Thus $I=S_{\alpha +1}$, the $\alpha +1$-st socle of $L_{K}(E)$. \ Since $%
S_{\alpha }$ and \ $S_{\alpha +1}/S_{\alpha }=Soc(L_{K}(E)/S_{\alpha })$ are
each graded ideals, so is $S_{\alpha +1}$ by a previous observation.  But then $V_{\alpha + 1} = E^0 \cap I$ by construction.


 Thus we have verified (1) and (2)  in  the induction process whenever $\gamma$ is not limit ordinal.

On the other hand, suppose $\gamma $ is a limit ordinal.
Then by definition $V_{\gamma }=\cup _{\alpha <\gamma }V_{\alpha }$, and $S_{\gamma } = \cup
_{\alpha <\gamma }S_{\alpha }$.   Since each $S_\alpha$ is a graded ideal, so is  $S_{\gamma}$.   It is then immediate that $V_{\gamma
}=E^{0}\cap S_{\gamma }$.

  Thus we have established (1) and (2) by transfinite induction.  But then (3) follows from (1) and (2) together with  Theorem \ref{Tomfordetheorem}(ii).   Finally,    (4)  follows immediately from (2), since for any ideal $I$ of $L_K(E)$, $I=L_K(E)$ if and only if $I\cap E^0 = E^0$.
\end{proof}


We can in fact glean from the proof of Theorem \ref{Loewyring} some additional information about  the individual members of the ascending Loewy socle series of a Leavitt path algebra.

\begin{proposition}\label{VNRegular}
Each $S_{\alpha }$ is a von Neumann regular ring.
\end{proposition}

\begin{proof}
  It is known (see e.g. \cite[pages 65, 90]{J}) that if $R$ is a semiprime ring
(i.e., $R$ has no nonzero nilpotent ideals), then the socle $Soc(R)$ is a
direct sum of simple rings $T_{i}$, each of which is a direct sum of isomorphic
simple left ideals and, moreover, each $T_{i}$ is a directed union of full
matrix rings over division rings. In particular each $T_{i}$, and hence $Soc(R)$, is von Neumann regular. Thus, by Proposition \ref{zeroradicalingeneral}, for the Leavitt path algebra $L_{K}(E)$,
its socle $S_{1}$ is always von Neumann regular. From the proof of Theorem \ref{Loewyring}, we notice the $\alpha$-th socle $S_{\alpha }$ is generated by the
hereditary saturated set $V_{\alpha }$ and that, by Theorem \ref{Tomfordetheorem}(ii),  $L_{K}(E)/S_{\alpha }$ is again a Leavitt path algebra (as it is isomorphic to $L_{K}(E|V_{\alpha })$). This implies that   $L_{K}(E)/S_{\alpha }$ is
semiprime since, by Proposition \ref{zeroradicalingeneral}, $J(L(E|V_{\alpha }))=0$. Consequently, $S_{\alpha +1}/S_{\alpha }=Soc(L_{K}(E)/S_{\alpha })$ is von Neumann regular
for all $\alpha \geq 1$. Since an extension of a von Neumann regular ring by
another von Neumann regular ring is again von Neumann regular, and since von
Neumann regularity survives under ascending unions, we conclude, by
transfinite induction, that each $S_{\alpha }$ is von Neumann regular.

\end{proof}

Recall that   a \emph{$K$-matricial algebra} is a finite direct product of full matrix algebras over $K$, while a \emph{locally $K$-matricial algebra} is a direct limit of $K$-matricial algebras.
As a consequence of Proposition \ref{VNRegular},  \cite[Theorem 1]{AR} then gives

\begin{corol}\label{Loewyimpliesacyclic}
Let $E$ be an arbitrary graph. If $L_{K}(E)$ is a Loewy ring, then
necessarily $E$ must be acyclic and $L_{K}(E)$ must be locally $K$-matricial
(and, in particular, von Neumann regular).
\end{corol}

We note that an arbitrary Loewy ring $R$ need not be von Neumann
regular even if it is semiprime, as shown for instance in \cite[Example 2.2]{Bac}.

As Example \ref{uppyramids} yields the  acyclic graph $Q_{\omega}$ for which $\ell(L_K(Q_{\omega}))=0$, we see that the converse to Corollary \ref{Loewyimpliesacyclic} does not hold in general.   However,

\begin{corol}
Let $E$ be a graph for which $E^0$ is finite,  and $K$ any field.  The following are equivalent.

(1)  $L_{K}(E)$ is a Loewy ring.

 (2)  $E$ is acyclic.

  (3)  $L_K(E)$ is von Neumann regular.

  \medskip

 If in addition $E^1$ is also finite, then the previous conditions are equivalent to
 
 \smallskip
 
 (4)   $L_K(E)$ is  semisimple artinian.  (In particular, in this case we have  $\ell (L_K(E)) = 1$.)
\end{corol}

\begin{proof}
(1) implies (2) follows from Corollary \ref{Loewyimpliesacyclic}.  For (2) implies (1), since $E^0$ is finite and $E$ is acyclic then necessarily $E$ contains sinks.  But any sink is  necessarily a line point in $E$.   So $S_1 = Soc(L(E))\neq \{0\}$.  If $S_1 \neq L(E)$ then by Theorem \ref{Tomfordetheorem} the quotient $L(E) / S_1 \cong L(F)$ for some (necessarily acyclic) quotient graph $F$ of $E$, for which $|F^0| < |E^0|$.     An induction argument now gives the result.  The equivalence of (2) and (3) was established in \cite[Theorem 1]{AR}.

In case $E^1$ is also finite, the equivalence of conditions (2), (3), and (4) follows from the fact (see e.g. \cite[Proposition 3.5]{AAS}) that $L_{K}(E)$ in this situation is isomorphic to a finite
direct sum of full matrix rings over $K$.   That  $\ell (L_K(E)) = 1$ in this case then follows directly from (4).
\end{proof}

\section{\protect\bigskip Leavitt path algebras of prescribed Loewy length}\label{prescribedLoewylength}

In this final section we demonstrate, as a key consequence of Theorem \ref{Loewyring}, how
to construct  graphs $E$ for which the corresponding Leavitt
path algebras $L_{K}(E)$ are Loewy rings of arbitrarily  prescribed Loewy
length.

 The idea of the following construction is to build graphs for which the subsets $V_{\lambda}$ are well understood.  Intuitively, we do this by ensuring that for each $\alpha$, the quotient $L_K(E)/S_{\alpha}$ is  isomorphic to the Leavitt path algebra of either the graph  $P_0$ or the graph $P_1$. As noted previously, these two graphs are the prototypical graphs whose Leavitt path algebras are Loewy rings of length 1.   (In fact these two graphs are related to each other: clearly the single vertex of  $P_0$ is a sink in the usual sense, while the graph $P_1$ may naturally be viewed as an {\it infinite sink}.  See e.g. \cite[Definition 1.7]{AAPS}.)

\begin{theorem}\label{realizationtheorem}
For every ordinal $\lambda $ and any field $K$, there is an acyclic graph $P_\lambda$
for which $L_{K}(P_\lambda)$ is a Loewy ring of length $\lambda $.
\end{theorem}

\begin{proof}
\ For $\lambda =1$, choose $E=P_{1}$, the ``infinite line" graph of
Example \ref{infinitelineexample}.  Observe that $P_{1}$ is acyclic, and has countably many vertices,
each of which is a line point. Thus, in the notation of Theorem \ref{Loewyring}, $V_{1}=P_{1}^{0} = E^0$.

We now utilize an approach similar to the one used in Example \ref{downpyramids} to
construct transfinitely the graphs $P_{\alpha +1}$ from $P_{\alpha}$ for
various ordinals $\alpha$.   This construction agrees with the construction of the graphs $P_n$ of Example \ref{downpyramids} for $n$ finite.


Suppose $\gamma \geq 2$ is any ordinal and assume that the
graphs $P_{\alpha }$ have already been  defined for all $\alpha <\gamma $
in such a way that:

\smallskip

(1)  $P_{\alpha }$ is acyclic,

(2)  $P_{\alpha}$ is a subgraph of $P_{\alpha + 1}$ for all $\alpha + 1 < \gamma$, and

(3)  whenever $\alpha$ is not a limit ordinal,  $P_{\alpha - 1}^0$ (i.e., the vertices of $P_{\alpha - 1}$)
is a hereditary saturated subset of $P_\alpha$, and the quotient graph $P_{\alpha}|P_{\alpha - 1}^0$ is isomorphic as a graph to either the graph $P_0$ or the graph $P_1$ given in Examples \ref{acyclicexample} and \ref{infinitelineexample}.

\smallskip

There are three possibilities for $\gamma $.

\smallskip

First, if $\gamma $ is a limit ordinal, then define
$$P_{\gamma }=\cup _{\alpha
<\gamma }P_{\alpha }.$$

\medskip

Second,  suppose $\gamma =\alpha +1$, where $\alpha $ is a limit ordinal.  By
definition,  $P_{\alpha}$ is the union of $P_{\beta }$, $\beta <\alpha.$   We define the graph $P_{\gamma} = P_{\alpha + 1}$ by introducing new symbols $v_{\alpha + 1,1}$ and $\{e^{\alpha + 1}_{\beta}: \beta < \alpha \}$, and by setting:
$$P_{\gamma}^0 = P_{\alpha }^0 \ \ \cup \ \  \{v_{\alpha + 1,1}\}, \ \ \mbox{and} \ \
 P_{\gamma}^1 = P_{\alpha}^1 \ \ \cup \ \ \{e^{\alpha +1}_{\beta} : \beta <\alpha\}$$
and by defining  $s(e^{\alpha + 1}_{\beta})=v_{\alpha
+1,1}$ and  $r(e^{\alpha + 1}_{\beta})=v_{\beta ,1}$.

\medskip

Pictorially,

$$\xymatrix{ P_{\gamma} = P_{\alpha + 1} = P_{\alpha} \ \cup   &
               &  \\
            & {\bullet}^{v_{\alpha + 1,1}}
            \ar @{=|>} [u]^{\{e^{\alpha + 1}_{\beta}\}}  }$$
where the double arrow indicates multiple edges indexed by $\beta < \alpha$, with $r(e^{\alpha + 1}_{\beta}) = v_{\beta ,1}$ for each $\beta$.

\bigskip

For the third possibility, suppose $\gamma =\alpha +1$, where $\alpha $ is not
a limit ordinal.  We define $P_{\gamma} = P_{\alpha + 1}$ by introducing new symbols    $\{v_{\alpha + 1, n} : n\in \N\}$, $\{e_{\alpha ,n}: n\in \N\}$, and $\{f_{\alpha + 1 ,n} : n\in \N\}$, and by     setting:
$$P_{\gamma}^0 = P_{\alpha }^0 \ \cup \ \{v_{\alpha + 1, n} : n\in \N\}, \ \ \mbox{and}$$
  $$P_{\gamma}^1 = P_{\alpha }^1 \ \cup \ \{e_{\alpha ,n}:n\in \N\} \ \cup \ \  \{f_{\alpha + 1 ,n} : n\in \N\},$$
where $s(e_{\alpha ,n})=v_{\alpha +1,n}$, $r(e_{\alpha,n})=v_{\alpha,1 }$, $s(f_{\alpha + 1,n})=v_{\alpha
+1,n}$, and   $r(f_{\alpha + 1 ,n})=v_{\alpha +1,n+1}$  for all $n \in \N$.

Pictorially,

$$\xymatrix{ P_{\gamma} = P_{\alpha + 1} = P_{\alpha} \ \cup   & &   &  &
               &  \\
            & {\bullet}^{v_{\alpha + 1,1}} \ar[r]_{f_{\alpha+1,1}} \ar[u]^{e_{\alpha,1}} & {\bullet}^{v_{\alpha + 1,2}} \ar[r]_{f_{\alpha+1,2}} \ar[ul]^{e_{\alpha ,2}}  & {\bullet}^{v_{\alpha + 1,3}} \ar[r]_{f_{\alpha+1,3}} \ar[ull]^{e_{\alpha,3}} & {\bullet} \ar[ulll]  \hdots  & }$$
with $r(e_{\alpha ,n}) = v_{\alpha  ,1}$ for all $n\in \N$.

\medskip

By transfinite induction, the graphs $P_{\gamma}$ are now defined for
every ordinal $\gamma$. It is clear from our construction that each $%
P_{\gamma}$ is an acyclic graph.   Moreover, it is also clear from the construction
that each $P_\gamma$ satisfies the indicated conditions (2) and (3) above.   (Again, recall
that the saturated condition applies only for regular vertices; note that in the case where
$\alpha$ is a limit ordinal, the vertex $\{v_{\alpha + 1, 1}\}$ is an infinite emitter,
and thus not a regular vertex.  Note also that in the graphs $P_\alpha$ the ``breaking vertex" sets $B_H$ are empty for all germane subsets $H$.)

\smallskip

Thus for each $\lambda$ the graph $P_\lambda$ has the property that $V_\lambda = P_\lambda^0$, and so $P_\lambda$ is a graph of the desired type by Theorem \ref{Loewyring}(4).

\end{proof}

Recall that in Section \ref{ExamplesSection}  we constructed, for each $n\leq \omega $, a row-finite graph $P_{n}$ for which $L(P_{n})$ is a Loewy ring of length $n$. We finish this article by showing that, in the row-finite case, length $\omega$ is the maximum possible.

\begin{theorem}\label{rowfinitegiveslengthleqomega}
If $E$ is a row-finite graph, then $L_{K}(E)$
must  have Loewy length $\leq \omega $.
\end{theorem}

\begin{proof}
Suppose, by way of contradiction, that $L_{K}(E)$ has length $>\omega $.
Let $S_{\omega }$ be the $\omega $-th socle of $L_{K}(E)$. By Theorem \ref{Tomfordetheorem}(iii),  $S_{\omega }$ is the ideal of $L_K(E)$
generated by $V_{\omega }=\cup _{n<\omega}V_{n}$.  For the same reason,   $S_{\omega + 1}$ is the
ideal of $L_K(E)$ generated by $V_{\omega +1}=V_{\omega }\cup Z$, where, in   general, $Z$ is the union of two subsets of vertices (recall Definition \ref{Valphadefinition}).  However, since $E$ is row-finite, the second of these two subsets is empty, so that here we have
$$
Z=\{v\in E^{0}\backslash V_{\omega }\text{ : every bifurcation vertex }u\in
T_{E}(v)\backslash V_{\omega }\text{ has at most one edge }e\in E^{1}$$
$$\hskip-2in \text{ with
}s(e)=u\text{ and }r(e)\notin V_{\omega}\}\text{.}
$$

\smallskip
\noindent
Let $v\in Z$ and let $\{u_{i}:u_{i}\in T_{E}(v)\backslash V_{\omega },i\in
X\}$ be the set of all bifurcation vertices in $T_{E}(v)\backslash V_{\omega
}$. For a given $i\in X$, let $s^{-1}(u_{i})=\{e_{i_1},...,e_{i_{k_{i}}}\}$, and let
$$J_i=\{e_{i_j} \in s^{-1}(u_{i}) :r(e_{i_j})\in V_{\omega }\}.$$
 Note that, by the conditions on $v\in Z$,  $|J_i|=k_{i-1}$ or $k_{i}$. Since $V_{\omega }$ is the union of the ascending
chain $V_{1}\subset V_{2}\subset \hdots$, there is a positive integer $m$
such that $r(J_i) \subseteq V_{m}$.  Thus $u_{i}$
is a line point in $E|V_{m}$.  But  by Theorem \ref{Loewyring}(3) we have $L_{K}(E)/S_{m}\cong L_{K}(E|V_{m})$, and so $u_{i}$ maps into
the socle of $L_{K}(E)/S_{m}$. This means that
$$u_{i} \in S_{m+1}\cap
E^{0}=V_{m+1}\subset V_{\omega }.$$ This contradicts the fact that $u_{i}\in
T_{E}(v)\backslash V_{\omega }$. Hence $L_{K}(E)$ must have length $\leq
\omega $.
\end{proof}

\end{document}